\documentclass[twoside,11pt]{amsart} 
\usepackage{epsfig}
\usepackage{latexsym} 
\usepackage{amsmath} 
\usepackage{amssymb}

\newtheorem{theorem}{Theorem}[section]
\newtheorem{lemma}[theorem]{Lemma}

\newdimen\epsfxsize

\newcommand {\gap}     {\makebox[0.075 in]{}}

\newcommand{\cons}[1] {\left<{#1}\right>}
\newcommand{\R}{\mathbb{R}}
\newcommand{\ER} {1 \times \R^d}
\newcommand {\set}[1]  {\left\{ {#1} \right\}}
\newcommand{\Wdg} {\omega}

\newcommand{\PWdg}{\tilde{\omega}}

\begin{document}

\title{Observations on the perturbed wedge}
\author{Fred B. Holt}
\address{5520 - 31st Ave NE; Seattle, WA 98105; 
{\rm fbholt@earthlink.net}}

\date{posted ArXiv 3 Nov 2013}

\begin{abstract}
Francisco Santos has described a new construction, perturbing apart a non-simple face, to offer
a counterexample to the Hirsch Conjecture.  We offer a series of observations about the initial spindle and about the {\em perturbed wedge} construction.

\vskip .0325in

NOTE:  These are simply working notes, offering a series of observations on the
 construction identified by Santos.
\end{abstract}


\maketitle


\section{Introduction}
Coincidence in high dimensions is a delicate issue.  Santos has brought forward the
perturbed wedge construction \cite{Paco}, to produce a counterexample to
 the Hirsch conjecture.  We start in dimension $5$ with a non-simple 
 counterexample to the nonrevisiting conjecture.  If this counterexample were simple, 
 then repeated wedging would produce a corresponding counterexample to the 
 Hirsch conjecture.  Since our counterexample to the nonrevisiting conjecture is
 not simple, we need an alternate method to produce the corresponding counterexample to
 the Hirsch conjecture, and the perturbed wedge provides this method.

A $d$-dimensional spindle $(P,x,y)$ is a polytope with two distinguished 
vertices $x$ and $y$ such that every facet of $P$ is incident to either $x$ or $y$.
The {\em length} of the spindle $(P,x,y)$ is the distance $\delta_P(x,y)$.
The spindle $(P,x,y)$ is {\em all-but-simple} if every vertex of $P$ other than $x$ and $y$ is 
a simple vertex.

Our first observation is that a $d$-dimensional all-but-simple spindle $(P,x,y)$ 
of length $d+1$ is a counterexample to the nonrevisiting conjecture.

Let $P$ be a $d$-dimensional polytope with $n$ facets.  Let $y$ be a nonsimple vertex of $P$ incident to a facet $G$, and let $F$ be a facet not incident to $y$.  The {\em perturbed wedge}
$\tilde{\omega}_{F,G}P$ of $P$ is
a wedge of $P$ with foot $F$, which is a $(d+1)$-dimensional polytope with $n+1$ facets,
followed by a perturbation of the image of the facet $G$ in the coordinate for the new 
dimension.

Our second set of observations focus on the structure of nonsimplicities and how the image of
a nonsimple vertex progresses under iterated applications of the perturbed wedge.

\section{A counterexample to the nonrevisiting conjecture}

\begin{lemma} Let $(P,x,y)$ be a $d$-dimensional all-but-simple spindle of length $d+1$.
Then every path from $x$ to $y$ revisits at least one facet.
\end{lemma}

\begin{proof}
Let $X$ be the $n_1$ facets incident to $x$, and let $Y$ be the $n_2$ facets incident to $y$.
Let $\rho = [x,u_1,\ldots,u_{k-1},y]$ be a path from $x$ to $y$ of length $k > d$, with all of the
$u_i$ being simple vertices.

Then each $u_i$ is incident to $d$ facets.  $u_1$ is incident to $d-1$ facets in $X$ and one
facet in $Y$, and $u_k$ is incident to one facet in $X$ and $d-1$ facets in $Y$.  The incidence
table for $\rho$ looks like the following:
$$
\begin{array}{llcccccc|cccccc}
 & \multicolumn{6}{c}{X} & \multicolumn{6}{c}{Y}  \\ \hline
x: & n_1 \; {\rm vertices} & 1 & 1 & \cdots & 1 & \cdots  & 1 & & & & & &  \\
u_1: & d \; {\rm vertices}  &  &  &  & 1 & \cdots & 1 & 1 & & & & &  \\
  & &  & &  &  & \cdots &  & & \cdots & & & &  \\
u_{k-1}: &  d  \; {\rm vertices}  & &  &  &  &  & 1 & 1 & \cdots & 1 &  &  &   \\
y: &  n_2 \; {\rm vertices}  & &  &  &  &  &  & 1 & \cdots & 1 & \cdots & 1 & 1  \\ \hline
\end{array}
$$

Consider the facet-departures and facet-arrivals from $u_1$ to $u_k$ (the simple part of the path).  
If any of the arrivals were back
to a facet in $X$, this would be a revisit since these facets were all incident 
to the starting vertex $x$.  So the arrivals must all be in $Y$.  

Similarly, all of the departures must be from $X$.  Any departure from a facet in $Y$ would
create a revisit since all the facets in $Y$ are incident to the final vertex $y$.

There are too many arrivals and departures to prevent a revisit.
The vertex $u_1$ is incident to only  $d-1$ facets in $X$, and since each departure leaves
a facet of $X$, $u_j$ is incident to $d-j$ facets in $X$.  So $u_d$ has completely departed
from $X$.  Since $k > d$, $u_d$ occurs among the vertices $u_1, \ldots, u_{k-1}$, but since
$u_d$ is incident only to facets in $Y$, $u_d$ must already be the vertex $y$, and $\rho$ would 
have length at most $d$.
\end{proof}

Santos \cite{Paco} has produced all-but-simple spindles in dimension $5$ of length $6$. 
Currently the smallest example has $25$ facets.
So in dimension $5$ with $25$ facets, we have a counterexample to the nonrevisiting conjecture, with length well below the Hirsch bound. Since the spindle is not simple, at $x$ and $y$, 
the usual way of generating the corresponding counterexample \cite{KW, Hnr} to the Hirsch conjecture (through repeated wedging) does not directly apply, and we need an alternate construction.  

\subsection{A gap in the literature - }
Santos' work has highlighted a gap in the previous literature on the Hirsch conjecture, regarding the equivalence of the simple and nonsimple cases.  We have the previously established results:

\begin{enumerate}
\item If $P$ is a nonsimple counterexample to the Hirsch conjecture of dimension $d$ and $n$ facets, then there is a simple polytope $Q$ also of dimension $d$ and $n$ facets that is also a counterexample to the Hirsch conjecture.  $Q$ is obtained directly from $P$ by perturbing apart its nonsimple vertices.
\item If $P$ is a simple polytope of dimension $d$ and $n$ facets, with two vertices $x$ and $y$ such that every path from $x$ to $y$ includes a revisit to a facets, then there is a simple polytope $Q$ of dimension $n-d$ and $2n-2d$ facets that is a counterexample to the Hirsch conjecture.  $Q$ is obtained directly from $P$ by taking wedges over all (the images of) the $n-2d$ facets of $P$ not incident to either $x$ or $y$.
\item Any counterexample to the Hirsch conjecture is also a counterexample to the nonrevisiting conjecture.
\end{enumerate}

The interesting gap that Santos has addressed is this:  suppose we have a nonsimple counterexample to the nonrevisiting conjecture, then how do we produce the corresponding counterexample to the Hirsch Conjecture?

The $5$-dimensional all-but-simple spindle $(P,x,y)$ produced by Santos is a counterexample
to the nonrevisiting conjecture, as we have verified above, but there are no facets not incident
to either $x$ or $y$.  So we cannot apply the wedges as we would if $P$ were simple.

We also cannot arbitrarily perturb apart the facets incident to $x$ and $y$, to create a simple 
polytope that is a counterexample to the nonrevisiting conjecture.  The revisit on a path may 
well occur along the last edge, the one terminating in $y$.  If we perturb the facets incident to 
$y$, to create simple vertices, some of these revisits may be lost.

Santos has offered the perturbed wedge as an equivalent construction in working with spindles.
We shift our attention now to analyzing the combinatorics of the perturbed wedge.

\section{The perturbed wedge}
The perturbed wedge is constructed in two steps, first as a wedge over a facet, followed
by a perturbation of a facet incident to a nonsimple vertex of the wedge.

Let $P$ be a $d$-dimensional polytope with $n$ facets and $m$ vertices, and let $F=F(u)$ be 
a facet incident to $x$ in $P$.
The wedge $W = \omega_F(P)$ is a $(d+1)$-dimensional polytope with $n+1$ facets
and $2m - f_0(F)$ vertices.
The wedge $\Wdg_F P$ over $F$ in $P$, corresponds to the two-point
suspension over $u$ in $P^*$.

\begin{figure}[tb] 
\centering
\includegraphics[width=4.5in]{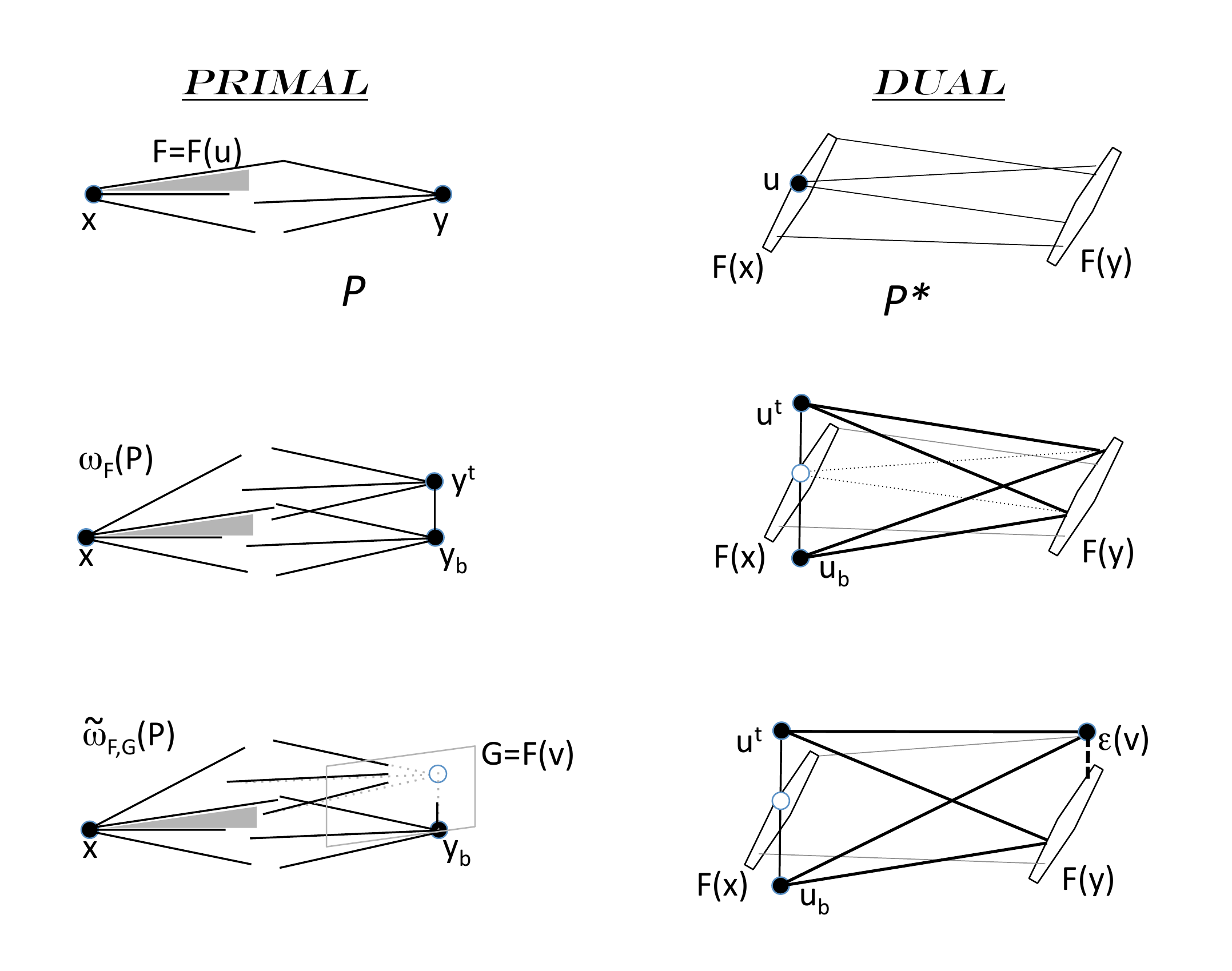}
\caption{\label{SantosWedgeFig} This figure illustrates the Santos perturbed wedge 
construction and its dual construction.  In the primal setting, we first perform a wedge
of $P$ over a facet $F$ which is incident to $x$, followed by a vertical perturbation of
a facet $G$ incident to $y$.}
\end{figure}

We now perturb a facet $G=F(v)$ in $W$ incident to the edge $[y^t, y_b]$.  This is already
interesting; we don't encounter non-simple edges until dimension $4$.
The perturbation of $G$ is accomplished by introducing a small vertical displacement
to $v$; that is, instead of the last coordinate of the outward-pointing normal vector to $G$ 
being $0$, we perturb this coordinate to $\epsilon > 0$.

\subsection{Embedded construction}
As a canonical embedding for polytopes, we consider the vertices of $P$
to be embedded in $\ER$, with $0$ in the interior of $P$.  For the facets
of $P$ we take their outward-pointing normals.  Since $0$ is interior to $P$,
we can assume that the first coordinate of each normal is $-1$.  $P$ is
given by the embedding:
$$ H^T_{n \times (d+1)} V_{(d+1) \times m}  \le \cons{0},$$
with
$$ \left[ -1 \gap h_i^T \right] \cdot \left[ \begin{array}{c} 1 \\ x_j \end{array} \right] = 0$$
iff vertex $j$ is incident to facet (hyperplane) $i$.
The facet-vertex incidence matrix for $P$ is given by the $\{0,1\}$-matrix
$$M_{n \times m}(P) = {\rm IsZero} \left( H^T V \right).$$

As a $d$-dimensional polytope, each vertex of $P$ is incident to at least
$d$ facets, and each $k$-face of $P$ is incident to at least $k+1$ vertices.
A simple vertex is incident to exactly $d$ facets.

For a $k$-face of $P$, each incident facet contributes either to the affine space
supporting this face or to its boundary \cite{Hblend}.
The space supporting this $k$-face is the intersection of at least $d-k$ facets of $P$,
and we say that the space is {\it simple} iff this space is given by the
coincident intersection of exactly $d-k$ facets of $P$. 
For $k > 0$, the boundary of the $k$-face is created by the various intersections of at least $k+1$
other facets with the supporting space of the face.
A face is simple iff its space is simple and all of its boundary elements are simple.
A face can be nonsimple in a variety of ways or in multiple ways, through the nonsimplicity
of its space or of its various boundary elements.

Let $F$ be represented by the first outward-pointing normal in $H^T$, and $G$ by the
last outward-pointing normal.  Then $H^T(\Wdg_F P)$ is given canonically by
$$ H^T(\Wdg_F P) = \left[  \begin{array}{ccc}
-1 & h_1^T & 1 \\
-1 & h_1^T & -1 \\
-1 & h_2^T & 0 \\
\vdots & \vdots & \vdots \\
-1 & h_n^T & 0 \end{array} \right].$$
The facet $F$ is replaced by two facets, the top and the base of the wedge.
The top has final coordinate $1$, and the base $-1$.
Every other facet is replaced by a single {\it vertical} facet, meaning that the last
coordinate (the new coordinate) is $0$.

The vertices of $\Wdg_F P$ are given as follows.  Rearrange the columns of $V$
so that the vertices incident to $F$ are given in the first block $V_F$ and
the rest of the vertices occur in a second block $V_{-}$ 
(denoted this way since $[-1 \; h_1^T]\cdot V_{-} < \cons{0}$).
\begin{eqnarray*}
V(P)  & =  & \left[ \begin{array}{cc} V_F & V_{-} \end{array} \right]. \\
  {\rm and} & &  \\
 V(\Wdg_F P) & = & \left[ \begin{array}{ccc}
 V_F & V_{-} & V_{-} \\
 \cons{0} & -[-1 \; h_1^T]\cdot V_{-} & [-1 \; h_1^T]\cdot V_{-}
 \end{array} \right].
 \end{eqnarray*}
 In $V(\Wdg_F P)$, the vertices are now embedded in $1 \times \R^{d+1}$.
 The last coordinate for vertices in the foot $F$ is $0$, and for vertices not in the
 foot, there are two images, one in the top and one in the base.

Let the facet $G$ have outward-pointing normal $[-1 \; h_n^T \; 0]$.  We perturb
the last coordinate to $\epsilon > 0$ to complete the construction of the perturbed
wedge.
$$ H^T(\PWdg_{F,G} P) = \left[  \begin{array}{ccc}
-1 & h_1^T & 1 \\
-1 & h_1^T & -1 \\
-1 & h_2^T & 0 \\
\vdots & \vdots & \vdots \\
-1 & h_{n-1}^T & 0 \\
-1 & h_n^T & \epsilon \end{array} \right].$$
Denote the outward-pointing normal for $\tilde{G}$ by 
$h_{\tilde{G}}^T = [-1 \; h_n^T \; \epsilon].$

To understand the effect of perturbing the facet $G$, we consider both $G$ and its
perturbed image $\tilde{G}$.  While we were able to write down the vertices of the
wedge $\Wdg_F P$ explicitly, the effect of the perturbation is more complicated.
The vertices incident to the facet $G$ are of three types:
\begin{itemize}
\item[Foot:] $v \in G \cap T \cap B$ (or $v \in G \cap F$).  For these vertices the last coordinate
is $0$, so $h_{\tilde{G}}^T v = 0$, and these vertices remain after the perturbation.

\vskip .0625in

\item[Top:] $v \in G \cap T  \backslash B$.  For these vertices, the last coordinate is positive,
so $h_{\tilde{G}}^T v > 0$.  If $v$ consists combinatorially of a single edge terminated by
$G$ -- the case when $v$ is a simple vertex but also when $v$ consists of a single 
nonsimple edge terminated by $G$ -- then the vertex $v$ is perturbed back along this
edge.  If $v$ consists combinatorially of more than one edge being terminated by $G$,
then $v$ is truncated away by the perturbation, and $\tilde{G}$ introduces vertices along
all of the edges incident to $v$ but not lying in $G$.

\vskip .0625in

\item[Base:] $v \in G \cap B  \backslash T$.  For these vertices, the last coordinate is negative,
so $h_{\tilde{G}}^T v < 0$.  If $v$ consists combinatorially of a single edge terminated by
$G$ -- the case when $v$ is a simple vertex but also when $v$ consists of a single 
nonsimple edge terminated by $G$ -- then the vertex $v$ is perturbed out along this
edge.  If $v$ consists combinatorially of more than one edge being terminated by $G$,
then $v$ remains as a vertex of $\PWdg_{F,G}P$,
and the perturbation reveals new edges emanating from $v$ and terminating in 
$\tilde{G}$ in new vertices.
\end{itemize}

We now consider the effect of the perturbed wedge on vertex $y$ and its natural images.

Under the wedge, $y$ has two natural images $y^t$ and $y_b$, in the top and base respectively.
Since $y$ is a nonsimple vertex, $y^t$ and $y_b$ are nonsimple vertices, and the edge
$[y_b,y^t]$ between them is a nonsimple edge.  The facet $G$ is one of the facets supporting 
the space of this edge.

When $G$ is perturbed to $\tilde{G}$, $y_b$ is preserved as a vertex, $y^t$ is truncated away,
and $\tilde{G}$ terminates the vertical edge at a new vertex $y_0$ whose last coordinate is $0$.
$\tilde{G}$ introduces new vertices along the edges of $\Wdg_F P$ incident to $y^t$ but not
lying in $G$.  $\tilde{G}$ also introduces new vertices along the new edges emanating from $y_b$
as revealed by $\tilde{G}$.  That is, the collection of facets $Y$ and the facet $B$ intersect in
edges that had lain beyond the facet $G$.  $\tilde{G}$ now introduces these edges as part of the
boundary of $\PWdg_{F,G}P$ and terminates them in new vertices.

In considering the implications of the perturbed wedge construction on the Hirsch conjecture,
we are interested in short paths from $x$ to $y$ in $P$ and their tight natural images
from $x$ to $y_0$ in $\PWdg_{F,G}P$.

{\em Claim:} The perturbed wedge does not introduce revisits on tight natural images of short paths.  
From our study of this construction, we are not seeing the mechanism that would introduce revisits,
but this note does not prove the absence of this mechanism.

For the construction of the counterexample to the Hirsch conjecture, the revisits already exist 
in the initial $5$-dimensional spindle.  We see below that as a general construction, the perturbed wedge does not always increase the length of the input polytope by $1$.  
However, the repeated application of the perturbed wedge to an all-but-simple spindle avoids the conditions of the following lemma.

\begin{figure}[tb] 
\centering
\includegraphics[width=4.5in]{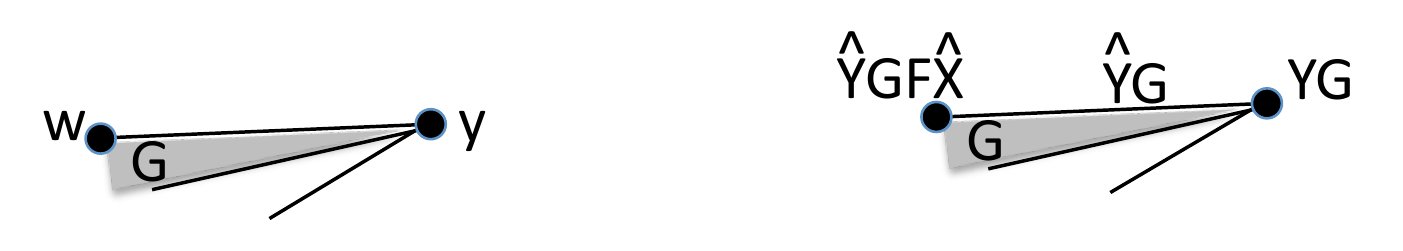}
\caption{\label{FacetFig} Consider the facet incidences in the circumstance that
$y$ has a neighbor $w$ along an edge of $G$ that terminates in the facet $F$, which
will be the foot of the wedge.  Denote the facet incidences at the nonsimple vertex $y$ 
as $YG$, in which $Y$ is a set of at least $d$ facets.  The edge from $y$ to $w$ is the
coincidence of $\hat{Y}G$ in which at least one facet of $Y$ is omitted from $\hat{Y}$.
The vertex $w$ is coincident with the facets $\hat{Y}G$ and $F$ and perhaps additional
facets $\hat{X}$.  $\hat{X}$ may be empty.}
\end{figure}

\begin{lemma}  Let $y$ be a nonsimple vertex of a $d$-dimensional polytope $P$.  Let $F$
be a facet of $P$ not incident to $y$, and let $G$ be a facet incident to $y$.
 If there is a nonsimple edge in $G$ from $y$ to a vertex $w$ in $F \cap G$, then this edge 
 remains after the perturbation.
\end{lemma}

\begin{figure}[tb] 
\centering
\includegraphics[width=4.5in]{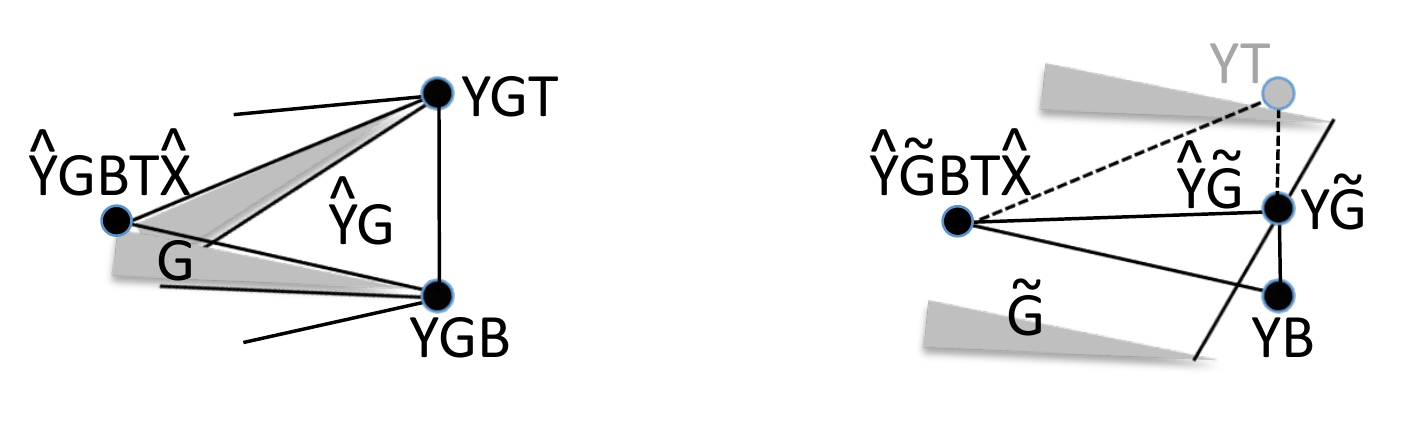}
\caption{\label{NbrVtxFig} Consider the action of the perturbed wedge on the
facet-coincidences at $y$ and $w$.  Under the wedge, $YG$ becomes an edge and
$\hat{Y}G$ becomes a $2$-face.  The vertex $y$ has two natural images, $y^t$ which is 
incident to the facets $YGT$, and $y_b$ which is incident to $YGB$.  The facet $F$ is 
replaced by two facets, the top $T$ and the base $B$.
Now we perturb the facet $G$, introducing a small positive value in the last coordinate of its outward
normal.
Since $y$ was nonsimple, the vertex $YB$ remains, but the vertex $YT$ is truncated away.
Instead, $\tilde{G}$ now intersects the vertical edge $Y$ in the plane 
$\set{1}\times \R^d \times \set{0}$.  The $2$-face $\hat{Y}G$  was nonsimple, and so the
$2$-face remains with its space supported by $\hat{Y}$, and $\tilde{G}$ intersects
it in an edge $[w,y_0]$.
}
\end{figure}

\begin{proof}
As a nonsimple edge, the 1-dimensional space of $[w,y]$ is defined by 
the coincidence of $G$ and at least $d-1$ other facets $\hat{Y}$.  
See Figure~\ref{FacetFig}. 
These facets $\hat{Y}G$
are incident to both $w$ and $y$.  The boundary of the edge at $w$ is established by
$F$ and possibly more facets $\hat{X}$, none of which can be incident to $y$.  The boundary
of the edge at $y$ is established by the facets $Y\setminus \hat{Y}$.

Under the wedge $\Wdg_F P$, the image of the edge is a nonsimple $2$-dimensional face,
the triangle with vertices $w$, $y^t$, and $y_b$.

Now, under the perturbation, the space of the triangular face is still defined by $\hat{Y}$.
$\tilde{G}$ intersects this face in the plane with last coordinate $0$, creating an edge from
$w$ to $y_0$.
\end{proof}


In the dual setting, the wedge over the facet $F(u)$ corresponds to a
two-point suspension $S_u(P^*)$ over the vertex $u$.  See Figure~\ref{SantosWedgeFig}.
The perturbation of the facet $G=F(v)$ corresponds to a vertical perturbation of the vertex $v$.
The observation in the previous lemma is that if the facet $Y$ is adjacent to a facet $W$, such
that $u,v \in W$ and that the ridge $Y \cap W$ is not simplicial, then after the two-point suspension
over $u$ and the perturbation of $v$, the new $Y$ is still adjacent to $W$ across the ridge, with
$v$ removed from the ridge.


\section{Technical notes on the proof of Santos Theorem 2.6}

In the proof of Theorem 2.6 in \cite{Paco}, there are a couple of assertions that have
interesting technical details regarding simplicity and nonsimplicity.

Theorem 2.6 is stated in the dual setting, and the second half of the proof of Theorem 2.6 describes
the perturbation.  Perturbing a vertex in the dual setting is equivalent  to perturbing the
corresponding facet in the primal setting.  So in selecting a facet to be perturbed (our facet $G$
in the construction above), Santos selects the corresponding vertex $a$:

\begin{center}{"...let $a$ be a vertex of $Q^+$..."} \end{center}

Initially, the spindle is all-but-simple, so the nonsimplicity is concentrated at the two special
vertices $x$ and $y$.  As we iteratively apply wedging, we create nonsimple edges, $2$-faces, and
so on.   There is structure to the nonsimplicity at the images of $x$ and $y$, and we have to
consider how this affects our choice of the facet (or its corresponding vertex in the dual) for 
perturbation.

\subsection{The structure of nonsimple faces}

A $k$-face $F^k$ of a polytope consists of a supporting $k$-dimensional space 
and a boundary.  A facet 
incident to $F^k$ contributes either to defining the supporting space or to defining the boundary.
So relative to a given $k$-face $F^k$, we can partition incident facets into two sets:  
the space-supporting facets and the boundary facets.

For a simple $k$-dimensional face, the supporting space is defined
by the intersection of exactly $d-k$ hyperplanes, and the boundary is given by the
intersection of at least $k+1$ additional hyperplanes with this supporting space. 
The intersections of the boundary hyperplanes with the space-defining hyperplanes and with
each other are all simple intersections.  (We note that vertices are exceptions.  The
$0$-dimensional space is given by the intersection of $d$ hyperplanes, but no additional
facets are required for the boundary.)

{\em Nonsimplicity}  Let $P$ be a $d$-dimensional polytope with facets $H^T$.
A {\em nonsimplicity} of dimension $k$ is a collection of $m$ facets $\tilde{H}^T$ in $H^T$
that satisfy the following conditions:
\begin{enumerate}
\item {\em Nonsimple intersection.} The intersection of the $m$ facets in $\tilde{H}^T$ 
is the supporting space for a $k$-dimensional face $F$ of $P$, with $m > d-k$.
\item {\em Maximality.}  No facet in $H^T \backslash \tilde{H}^T$ contains the supporting space
of $F$.
\item {\em Affine support.}  There is no facet $h^T$ in $\tilde{H}^T$ such that
the intersection of the facets $\tilde{H}^T \backslash h^T$ is more than $k$-dimensional.
\end{enumerate}

We call this last condition the requirement for {\em affine support} because in the dual setting,
we may know that the points $\tilde{H}$ are affinely dependent, but there may be points
in $\tilde{H}$ that do not contribute to this affine dependence.  That is, for any affine dependence
$\tilde{H} \cdot \alpha = 0$, the coordinate in $\alpha$ corresponding to $h$ is always $0$.
We will see examples of this below.

The {\it excess} of a nonsimplicity is the quantity $m-(d-k)$, the number of facets supporting
the face beyond those necessary in a simple polytope.

{\em Nonsimplicities in a nonsimple polytope.} 
A nonsimple polytope may contain several nonsimplicities, and these nonsimplicities
may be of different dimensions.  

Let's consider a few examples.  A pyramid over a 
hexagon is a $3$-dimensional polytope, whose apex $v$ is a nonsimple vertex,
but the spaces supporting all of the edges and $2$-faces are simple.  The prism over
this pyramid is a $4$-dimensional polytope whose only nonsimplicity is the edge
$[v_b, v^t]$.  The vertices $v_b$ and $v^t$ are not simple vertices, but in each case
the nonsimple vertex is the intersection of one facet with the nonsimplicity in the vertical edge.
So although the vertex $v^t$ is not simple, the facets incident to $v^t$ do not satisy the
condition of affine support.

For the next example, consider
our all-but-simple spindle $(P,x,y)$.  When we perform the wedge over a facet $F$ incident
to $x$, the nonsimple vertex $y$ generates a nonsimplicity in the vertical edge $[y_b, y^t]$.  
This edge contains the nonsimplicity, and its two vertices consist of the intersection of the
nonsimplicity supporting the vertical edge with the top facet $T$ and with the bottom facet $B$.  

Suppose we wanted to reduce the nonsimplicity of this edge $[y_b,y^t]$ by perturbing a facet.
If we perturb $T$ or $B$, the nonsimplicity is not changed.  We must perturb a facet that is
contributing to the nonsimplicity.  That is, to reduce the nonsimplicity of the edge $[y_b, y^t]$
we must perturb one of its space-supporting facets.

Continuing this example, if we now construct the wedge over a facet incident to the image of $x$ in the first wedge, the nonsimple edge $[y_b,y^t]$ generates a nonsimple $2$-face.  If the foot of this wedge is $T$ or $B$, then the image of $[y_b, y^t]$ is a triangle; otherwise it is a rectangle.
In either case the nonsimplicity occurs in the supporting $2$-space.  

The nonsimplicity of a nonsimple $k$-face can arise in any dimension from $0$ to $k$, and the 
nonsimplicity of this $k$-face may be a combination of various nonsimplicities across these dimensions.  So in choosing the facet to perturb as part of the construction of the perturbed wedge,
we need to choose a facet that is contributing to some nonsimplicity.  

This approach to nonsimplicity borrows heavily from the insights underlying the Gale transform
\cite{Gr, Zg}.

\subsection{ Implications for constructing the perturbed wedge.}
In the proof of Theorem 2.6 in \cite{Paco}, we select a facet to be perturbed by
selecting the corresponding dual vertex $a$:

\begin{center}{"...let $a$ be a vertex of $Q^+$..."} \end{center}

We now see that we cannot select just any vertex in $Q^+$.  We have to choose a vertex
that is contributing to the nonsimplicity.

In the primal setting, number the hyperplanes incident to the nonsimple vertex $y$
in the order in which they are perturbed or used
as the foot for wedging.

$$  \left[ \begin{array}{cc}
-1 & h_1^T \\
-1 & h_2^T \\
-1 & h_3^T \\
 \vdots & \vdots \\
-1 & h_m^T
\end{array} \right]  \cdot \left[ \begin{array}{c} 1 \\  y \end{array} \right] = \left[ \cons{0} \right]. $$

$$
\begin{array}{llll}
{\rm dim} P = d & {\rm dim} S = 0 & \#{\rm facets} = m & {\rm excess} = m-d
\end{array} $$

To begin with, we assume that the nonsimplicity is contained in the vertex itself and not in 
any incident faces.   After the first wedge, over a facet incident to the other nonsimple vertex $x$,
we perturb the first hyperplane incident to $y$.

$$ \left[ \begin{array}{ccc}
-1 & h_1^T & \epsilon_1 \\
-1 & h_2^T & 0 \\
-1 & h_3^T  & 0 \\
  &  \vdots &   \\
-1 & h_m^T & 0 
\end{array} \right]  \cdot \left[ \begin{array}{cc} 
1 & 0  \\  y & \cons{0} \\ 0 & -1  \end{array} \right] = \left[ \begin{array}{cc}
0 & -\epsilon_1 \\  
0 & 0 \\
 \vdots & \\
 0 & 0
\end{array} \right]. $$

$$
\begin{array}{llll}
{\rm dim} P = d+1 & {\rm dim} S = 1 & \#{\rm facets} = m-1 & {\rm excess} = m-d-1
\end{array} $$

The first column on the right hand side indicates that the vertex $y$ is still incident to all $m$ hyperplanes.
The second column indicates that the space supported by
$$ s_1 = \left[ \begin{array}{c} 0 \\ \cons{0} \\ -1 \end{array} \right] $$
contains an edge incident to hyperplanes $2 \ldots m$ and behind the first hyperplane.  
We know that this is a vector and not a point, since the first coordinate is $0$ (recall that we are
working in $1 \times \R^{d+1}$. 
So there is a nonsimple edge in $\PWdg(P)$ incident to $y$ in the direction $s_1$,
and this edge is incident to $m-1$ hyperplanes.  

It is the decreasing index of nonsimplicity, the excess, that Santos observes leads ultimately to 
a simple polytope.
From a nonsimple vertex in dimension $d$ incident to $m$ facets, we now have a nonsimple edge
in dimension $d+1$ incident to $m-1$ facets.

Let's work through a few more iterations of the perturbed wedge, to observe how the nonsimplicity progresses.  In the second iteration, we take the wedge over the second hyperplane.

$$ \begin{array}{c}
\left[ \begin{array}{cccc}
-1 & h_1^T & \epsilon_1 & 0 \\
-1 & h_2^T & 0 & -1 \\
-1 & h_2^T & 0 & 1 \\
-1 & h_3^T  & 0 & 0 \\
  &  \vdots &   &  \\
-1 & h_m^T & 0 & 0
\end{array} \right]  \cdot \left[ \begin{array}{ccc} 
1 & 0 & 0  \\  y & \cons{0} & \cons{0} \\ 0 & -1 & 0 \\ 0 & 0 & 1  \end{array} \right] = \left[ \begin{array}{ccc}
0 & -\epsilon_1 & 0 \\  
0 & 0 & -1 \\ 0 & 0 & 1 \\ 0 & 0 & 0 \\
 & \vdots & \\
 0 & 0 & 0
\end{array} \right].  \\ \\

\begin{array}{llll}
{\rm dim} P = d+2 & {\rm dim} S = 1 & \#{\rm facets} = m & {\rm excess} = m-d-1
\end{array}
\end{array} $$
Note that we now have $m+1$ hyperplanes in the left factor.  The right factor consists of the point
$y$ and two vectors $s_1$ and $s_2$.  What does the righthand side tell us?  The first column 
indicates that all $m+1$ hyperplanes are incident to the vertex $y$.  The second column
indicates that the edge supported by $s_1$ is incident to $m$ hyperplanes and lies behind
the first hyperplane, so this is a nonsimple edge of the polytope.  The third column tells us
that the space supported by $s_2$ is incident to all $m-1$ hyperplanes except the two natural 
images of $h_2^T$.  And this space lies behind one natural image of $h_2^T$ and beyond
the other.  

So the direction $s_2$ is pinched off by the wedge, and we have a polytope in dimension
$d+2$ with a vertex $y$ incident to $m+1$ facets and an incident edge containing the nonsimplicity.
The edge is incident to $m$ facets.  (As a reminder here, we are studying only the images
of the facets incident to $y$ in the original polytope.  Similar effects are occurring over at the
the vertex $x$.)


\underline{Case $\epsilon \rightarrow \omega$}.
Consider a variation on the wedge, in which we take the wedge over the same facet we 
previously perturbed.
$$ 
\begin{array}{c}
\left[ \begin{array}{cccc}
-1 & h_1^T & \epsilon_1 & -1 \\
-1 & h_1^T & \epsilon_1 & 1 \\
-1 & h_2^T & 0 & 0 \\
-1 & h_3^T  & 0 & 0 \\
  &  \vdots &   &  \\
-1 & h_m^T & 0 & 0
\end{array} \right]  \cdot \left[ \begin{array}{ccc} 
1 & 0 & 0  \\  y & \cons{0} & \cons{0} \\ 0 & -1 & 0 \\ 0 & 0 & 1  \end{array} \right] = \left[ \begin{array}{ccc}
0 & -\epsilon_1 & -1 \\  
0 & -\epsilon_1 & 1 \\   0 & 0 & 0 \\
 & \vdots & \\
 0 & 0 & 0
\end{array} \right]. \\ \\
 
\begin{array}{llll}
{\rm dim} P = d+2 & {\rm dim} S = 1 & \#{\rm facets} = m-1 & {\rm excess} = m-d-2; \\
    & {\rm dim} S = 0 & \#{\rm facets} = m+1 & {\rm excess} = m-d-1 \\
\end{array} 
\end{array}$$
The first column in the righthand side tells us that all of the facets are incident to the natural image
of $y$ (the first column in the second factor on the lefthand side).
The second and fourth columns give us two vectors supporting the nonsimplicity, with the natural 
images of $h_1^T$ and $h_3^T$ forming the boundary.
The third column tells us that the natural images of $h_2^T$, now the top and base of that iteration
of the wedge, support the $2$-dimensional nonsimplicity and together prevent the
nonsimplicity being $3$-dimensional.  This is an interesting case, in that there are two nonsimplicities:
the $1$-dimensional nonsimplicity supporting an edge of the polytope, and the vertex formed
by the intersection of this edge with $2$ additional hyperplanes.

Now consider a third iteration of the perturbed wedge.
In this third iteration, we perform a second wedge over a facet incident to $x$ and a 
second perturbation of a facet here at $y$.
$$ 
\begin{array}{c}
\left[ \begin{array}{ccccc}
-1 & h_1^T & \epsilon_1 & 0 & 0 \\
-1 & h_2^T & 0 & -1 & 0\\
-1 & h_2^T & 0 & 1 & 0\\
-1 & h_3^T  & 0 & 0 & \epsilon_3 \\
  &  \vdots &   &  & \\
-1 & h_m^T & 0 & 0 & 0 
\end{array} \right]  \cdot \left[ \begin{array}{cccc} 
1 & 0 & 0  & 0 \\  y & \cons{0} & \cons{0} & \cons{0} \\ 
0 & -1 & 0 & 0  \\ 0 & 0 & 1 & 0 \\ 0 & 0 & 0 & -1 \end{array} \right] = \left[ \begin{array}{cccc}
0 & -\epsilon_1 & 0 & 0 \\  0 & 0 & -1 & 0 \\ 0 & 0 & 1 & 0 \\
0 & 0 & 0 & -\epsilon_3 \\
 & \vdots & &  \\
 0 & 0 & 0 & 0
\end{array} \right].  \\ \\

\begin{array}{llll}
{\rm dim} P = d+3 & {\rm dim} S = 2 & \#{\rm facets} = m-1 & {\rm excess} = m-d-2
\end{array}
\end{array} $$

\underline {Case $\epsilon \rightarrow \epsilon$.}
A variation on this second perturbation, perturbing the previously perturbed facet:
$$ 
\begin{array}{c}
\left[ \begin{array}{ccccc}
-1 & h_1^T & \epsilon_1 & 0 & \epsilon_3 \\
-1 & h_2^T & 0 & -1 & 0\\
-1 & h_2^T & 0 & 1 & 0\\
-1 & h_3^T  & 0 & 0 & 0 \\
  &  \vdots &   &  & \\
-1 & h_m^T & 0 & 0 & 0 
\end{array} \right]  \cdot \left[ \begin{array}{cccc} 
1 & 0 & 0  & 0 \\  y & \cons{0} & \cons{0} & \cons{0} \\ 
0 & -1 & 0 & 0  \\ 0 & 0 & 1 & 0 \\ 0 & 0 & 0 & -1 \end{array} \right] = \left[ \begin{array}{cccc}
0 & -\epsilon_1 & 0 & -\epsilon_3 \\  0 & 0 & -1 & 0 \\ 0 & 0 & 1 & 0 \\
0 & 0 & 0 & 0 \\
 & \vdots & &  \\
 0 & 0 & 0 & 0
\end{array} \right].  \\ \\

\begin{array}{llll}
{\rm dim} P = d+3 & {\rm dim} S = 2 & \#{\rm facets} = m-1 & 
 {\scriptstyle DEGENERATE}
\end{array} 
\end{array}$$
This is a degenerate case.  We have a nonsimplicity supported in a $2$-dimensional face, 
with only one additional facet.  So these facets no longer define a vertex, but they define instead
define an edge.  In the dual setting, what we are observing is that if we perturb the same point 
twice, this only increases the dimension of the convex hull once; the dual vertices no longer form
a facet but only a ridge.

\underline{Case $\Wdg \rightarrow \epsilon$}.
Another variation, perturbing a facet that was the top or base of a previous wedge:
$$ 
\begin{array}{c}
\left[ \begin{array}{ccccc}
-1 & h_1^T & \epsilon_1 & 0 & 0 \\
-1 & h_2^T & 0 & -1 & \epsilon_3 \\
-1 & h_2^T & 0 & 1 & 0 \\
-1 & h_3^T  & 0 & 0 & 0 \\
  &  \vdots &   &  & \\
-1 & h_m^T & 0 & 0 & 0 
\end{array} \right]  \cdot \left[ \begin{array}{cccc} 
1 & 0 & 0  & 0 \\  y & \cons{0} & \cons{0} & \cons{0} \\ 
0 & -1 & 0 & 0  \\ 0 & 0 & 1 & 0 \\ 0 & 0 & 0 & -1 \end{array} \right] = \left[ \begin{array}{cccc}
0 & -\epsilon_1 & 0 & 0 \\  0 & 0 & -1 & -\epsilon_3 \\ 0 & 0 & 1 & 0 \\
0 & 0 & 0 & 0 \\
 & \vdots & &  \\
 0 & 0 & 0 & 0
\end{array} \right]. \\ \\

\begin{array}{llll}
{\rm dim} P = d+3 & {\rm dim} S = 2 & \#{\rm facets} = m-1 & {\rm excess} = m-d-2
\end{array} 
\end{array} $$

After four iterations of the perturbed wedge, applying perturbations and wedges on the
images of the original hyperplanes in order, we have
$$  \begin{array}{c}
\left[ \begin{array}{cccccc}
-1 & h_1^T & \epsilon_1 & 0 & 0 & 0 \\
-1 & h_2^T & 0 & 1 & 0 & 0 \\
-1 & h_2^T & 0 & -1 & 0 & 0 \\
-1 & h_3^T & 0 & 0 & \epsilon_3 & 0 \\
-1 & h_4^T & 0 & 0 & 0 & 1  \\
-1 & h_4^T & 0 & 0 & 0 & -1  \\
  & \vdots  & & & & \\
-1 & h_m^T & 0 & 0 & 0 & 0 
\end{array} \right]  \cdot \left[ \begin{array}{ccccc} 
1 & 0 & 0 & 0 & 0  \\  
y & \cons{0} & \cons{0} & \cons{0} & \cons{0} \\
 0 & -1 & 0 & 0 & 0 \\ 0 & 0 & 1 & 0 & 0 \\ 0 & 0 & 0 & -1 & 0 \\ 0 & 0 & 0 & 0 & 1 
  \end{array} \right] = \left[ \begin{array}{ccccc}
0 & -\epsilon_1 & 0 & 0 & 0 \\  
0 & 0 & -1 & 0 & 0 \\ 0 & 0 & 1 & 0 & 0 \\ 0 & 0 & 0 & -\epsilon_3 & 0 \\ 
0 & 0 & 0 & 0 & -1 \\ 0 & 0 & 0 & 0 & 1 \\ 0 & 0 & 0 & 0 & 0 \\
 & \vdots & & &  \\
 0 & 0 & 0 & 0 & 0
\end{array} \right]. \\ \\

\begin{array}{llll}
{\rm dim} P = d+4 & {\rm dim} S = 2 & \#{\rm facets} = m & {\rm excess} = m-d-2
\end{array} 
\end{array} $$
We have $m+2$ hyperplanes, all of which are incident to $y$.  The spaces $s_1$ and $s_3$,
given by the second and fourth columns of the right factor,
are both incident to  $m+1$ hyperplanes and behind the remaining one; which tells us
that the nonsimplicity is now $2$-dimensional, containing the vertex $y$ and supported by
$s_1$ and $s_3$.  The other two dimensions, in the directions of $s_2$ and $s_4$ are pinched
off by the wedges.  The $2$-dimensional nonsimplicity is incident to $m$ hyperplanes.

The wedge preserves the dimension of the nonsimplicity and adds one supporting facet, 
thereby also preserving the excess.  We tabulate the data for the nonsimplicity incident to the natural
image of $y_0$, including the dimension of the nonsimplicity ${\rm dim} S$ and the number of
facets supporting the space containing the nonsimplicity (tabulated under "$\# {\rm facets}$").
\begin{center}
\begin{tabular}{ccccc}
 \hline
 Iteration & ${\rm dim} P$ & ${\rm dim} S$ & $\#{\rm facets}$ & excess \\ \hline
 $0$ & $d$ & $0$ & $m$ & $m-d$ \\
 $1$ & $d+1$ & $1$ & $m-1$ & $m-d-1$ \\
 $2$ & $d+2$ & $1$ & $m$ & $m-d-1$ \\
 $3$ & $d+3$ & $2$ & $m-1$ & $m-d-2$ \\
 $4$ & $d+4$ & $2$ & $m$ & $m-d-2$ \\ 
 \multicolumn{5}{c}{$\cdots$} \\
 $2(m-d)-1$ & $2m-d-1$ & $m-d$ & $m-1$ & $0$ \\ \hline
 \end{tabular}
 \end{center}
The excess is finally $0$ -- that is, the faces at $y$ are simple -- on the $2(m-d)-1$ iteration.

The above work illustrates that there are some interesting technical nuances behind selecting the
facets to use in iterating the perturbed wedge.  In the early iterations, most of the facets at $y$ will 
produce the desired effect under perturbation; however, in later iterations, we have to be more careful
in selecting the facet to perturb.

\subsection{Restoring more general conditions}
Perhaps to navigate around these details around nonsimplicity, Santos \cite{Paco} suggests 
a corrective step in the proof of his Theorem 2.6:

\begin{quote}
"...let $a$ be a vertex of $Q^+$ and assume that the only non-simplicial facets of $S_v(Q)$
containing $a$ are $Q^+ * u$ and $Q^+ * w$ (if that is not the case, we first push $a$ to 
a point in the interior of $Q^+$ but otherwise generic, which maintains all the properties
we need and does not decrease the dual distances, by Lemma 2.2)" \end{quote}

%

The equivalent statement in the primal setting is that the only nonsimplicities are
the two vertices $y_b$ and $y^t$ and possibly the edge between them.  This is true for the first
two iterations of the perturbed wedge, but our work above illustrates that if we simply iterate
the construction in its simplest form (perturb one vertex vertically), the dimension of the
nonsimplicity grows.

So what insights can we provide about augmenting the perturbed wedge construction
by "push(ing) $a$ to a point in the interior of $Q^+$ but otherwise generic"?

For iterations of the perturbed wedge construction, we have been studying the effects at 
the one end $y$ of the spindle.  Let's consider applying the push during the third iteration.
Without the push, after three iterations we have the following equation, describing
the nonsimplicity at the natural image of $y$.

$$
\begin{array}{c}
\footnotesize
\left[ \begin{array}{ccccc}
-1 & h_1^T & \epsilon_1 & 0 & 0 \\
-1 & h_2^T & 0 & -1 & 0\\
-1 & h_2^T & 0 & 1 & 0\\
-1 & h_3^T  & 0 & 0 & \epsilon_3 \\
  &  \vdots &   &  & \\
-1 & h_m^T & 0 & 0 & 0 
\end{array} \right]  \cdot \left[ \begin{array}{cccc} 
1 & 0 & 0  & 0 \\  y & \cons{0} & \cons{0} & \cons{0} \\ 
0 & -1 & 0 & 0  \\ 0 & 0 & 1 & 0 \\ 0 & 0 & 0 & -1 \end{array} \right] = \left[ \begin{array}{cccc}
0 & -\epsilon_1 & 0 & 0 \\  0 & 0 & -1 & 0 \\ 0 & 0 & 1 & 0 \\
0 & 0 & 0 & -\epsilon_3 \\
 & \vdots & &  \\
 0 & 0 & 0 & 0
\end{array} \right].  \\ \\ 

\begin{array}{llll}
{\rm dim} P = d+3 & {\rm dim} S = 2 & \#{\rm facets} = m-1 & {\rm excess} = m-d-2
\end{array}
\end{array} 
$$

Remember that we work in the primal setting while Santos presents his proof in the dual setting.
So his $a$ is the natural image of our $h_3^T$.  Pushing $h_3^T$ corresponds to an additional
perturbation constrained by staying incident to the image of $y$ and maintaining distances to $x$.
The qualifying condition for $h_3^T$ is that the only nonsimple vertices incident to $h_3^T$
(before the perturbation $\epsilon_3$) are $y_b$ and $y^t$.

From the equation above, we note that before the perturbation $\epsilon_3$, the hyperplane
$h_3^T$ is supporting a $2$-dimensional nonsimplicity.  So there are at least three nonsimple
vertices incident to $h_3^T$.   The indicated perturbation reduces the 
dimension to $1$, so that {\em after} the perturbation $h_3^T$ is incident to two nonsimple
vertices.

If we perturb other facets, there is any interesting tension between the following conditions:
\begin{itemize}
\item maintaining the incidence at $y$;
\item perturbing apart the incidences at images of $y_b$;
\item maintaining distances from $x$.
\end{itemize}

It would be interesting to describe the details of such a combination of perturbations.


\section{Summary}
Santos' construction of the first known counterexample to the Hirsch conjecture, for bounded
polytopes, follows the strategy of first finding a counterexample to the nonrevisiting conjecture.
Santos constructs a $5$-dimensional all-but-simple spindle $(P,x,y)$ of length $6$, which is
a counterexample to the nonrevisiting conjecture.

For simple polytopes, if we had a counterexample to the nonrevisiting conjecture, we would
produce the corresponding counterexample to the Hirsch conjecture through repeated wedging,
over all the facets not incident to $x$ or $y$.  However, Santos $5$-dimensional spindle is not simple.  Every facet is incident to either $x$ or $y$,
so we need an alternate method to produce the corresponding counterexample to the Hirsch
conjecture.  Santos has offered the perturbed wedge to accomplish this.

In these working notes, we have offered some technical details regarding the nonsimplicities
under iterations of the perturbed wedge construction.


\bibliographystyle{alpha}

\bibliography{../Biblio/poly}

\end{document}